\documentclass[11pt]{article}
\usepackage [dvips]{graphics}
\usepackage[centertags]{amsmath}
\usepackage{amsfonts}
\usepackage{amssymb}
\usepackage{amsthm}
\usepackage{newlfont}
\usepackage{stmaryrd}

\theoremstyle{plain}
\newtheorem{thm}{Theorem}[section]
\newtheorem{cor}[thm]{Corollary}
\newtheorem{lem}[thm]{Lemma}
\newtheorem{prop}[thm]{Proposition}

\newtheorem{rem}[thm]{Remark}

\def\sqr#1#2{{\vcenter{\vbox{\hrule height.#2pt
              \hbox{\vrule width.#2pt height#1pt \kern#1pt \vrule
width.#2pt}
              \hrule height.#2pt}}}}
\def\dbR{{\mathbb{R}}}

\def\d{{d\over dt}}

\def\3n{\negthinspace \negthinspace \negthinspace }
\def\2n{\negthinspace \negthinspace }
\def\1n{\negthinspace }

\def\no{\noindent}

\def\bs{\bigskip}

\def\dim{\hbox{\rm dim$\,$}}

\def\({\Big (}
\def\){\Big )}
\def\[{\Big[}
\def\]{\Big]}

\def\be{\begin{equation}}
\def\bel{\begin{equation}\label}
\def\ee{\end{equation}}
\def\bea{\begin{eqnarray}}
\def\eea{\end{eqnarray}}
\def\bt{\begin{theorem}}
\def\et{\end{theorem}}
\def\bc{\begin{corollary}}
\def\ec{\end{corollary}}
\def\bl{\begin{lemma}}
\def\el{\end{lemma}}
\def\bp{\begin{proposition}}
\def\ep{\end{proposition}}
\def\br{\begin{remark}}
\def\er{\end{remark}}
\def\ba{\begin{array}}
\def\ea{\end{array}}
\def\bd{\begin{definition}}
\def\ed{\end{definition}}

\makeatletter
   
   \@addtoreset{equation}{section}
\makeatother

\begin{document}

\title{\bf Hill-type formula and Krein-type trace formula  for $S$-periodic solutions  in ODEs }\author{Xijun Hu\thanks{Partially supported
by NSFC(No.11131004) PCSIRT ( IRT1264) and NCET, E-mail:xjhu@sdu.edu.cn }  \quad Penghui Wang\thanks{ Partially supported by NSFC(No.11471189),   E-mail: phwang@sdu.edu.cn }    \\ \\
 Department of Mathematics, Shandong University\\
Jinan, Shandong 250100, The People's Republic of China\\
}

\maketitle
\begin{center}  Dedicate to Rou-Huai Wang's 90th birth anniversary
\end{center}
\begin{abstract}
The present paper is devoted to studying the Hill-type formula and Krein-type trace formula for ODE, which is a continuous work of our previous work for Hamiltonian systems \cite{HOW}. Hill-type formula and Krein-type trace formula are given by Hill at 1877 and Krein in 1950's separately. Recently, we find that there is a closed relationship between them \cite{HOW}. In this paper, we will obtain the Hill-type formula for the $S$-periodic orbits of the first order ODEs. Such a kind of orbits is  considered naturally to study the symmetric periodic and quasi-periodic solutions. By some similar idea in \cite{HOW}, based on the Hill-type formula, we will build up the Krein-type trace formula for the first order ODEs, which can be seen as a non-self-adjoint version of the case of Hamiltonian system.
\end{abstract}

\bs

\no{\bf 2010 Mathematics Subject Classification}:  34B09, 34C27, 34L05

\bs

\no{\bf Key Words}. Hill-type formula;  Krein-type trace formula ; Fredholm determinant; Hilbert-Schmidt operator

\section{Introduction}
In the present paper, we will study the Hill-type formula and Krein-type trace formula for $S$-periodic solutions for the first order ODE. Hill-type formula was introduced by Hill \cite{Hi} when he considered the motion of lunar perigee at 1877, and the Krein-type trace formula was built up by Krein \cite{Kr1, Kr2} in 1950's when he studied the stability of Hamiltonian systems. Although they appeared separately, there is a closed relationship between them. In fact,  the Krein-type trace formula could be derived by the Hill-type formula for $S$-periodic orbits in Hamiltonian systems.  Moreover, motivated by the Krein's original work, the Krein-type trace formula was used to study the stability problem in $n$-body problem, details could be found in \cite{HW1,HOW}. In the case of Hamiltonian system, the corresponding differential operators are self-adjoint, and the formulas for the first order ODE can be seen as a non-self-adjoint version.

 Let $M_m(\mathbb C)$ be the set of  $m\times m$ matrices on $\mathbb{C}^{m}$, and denote by
 \begin{eqnarray}\
 \mathfrak{B}(m)=C([0,T],M_m(C)),
 \end{eqnarray}
 the set of continuous path of $m\times m$ matrices on $[0,T]$.
We  consider the $n$-dimensional  first order ODE with the $S$-periodic boundary problem
\begin{eqnarray}
 \dot{x}(t)&=&D(t)x(t)\label{eq0.12},\\
   x(0)&=&Sx(T)\label{eq0.13},
\end{eqnarray}
where $S$ is an orthogonal  $n\times n$ matrix and $D\in\mathfrak{B}(n)$. It is natural to study the $S$-periodic solutions of the first order ODE when we study the symmetric periodic solutions and quasi-periodic solutions. In what follows, we always denote by $\gamma_D(t)$ the fundamental solution of the first order ODE (\ref{eq0.12}), that is,  $\dot{\gamma}_D(t)=D(t)\gamma_D(t)$ with $\gamma_D(0)=I_n$.

Consider $\d$ as the unbounded closed operator densely defined on $L^2(0,T;\dbR^n)$ with the domain
\begin{eqnarray*}
D_{S}=\Big\{z(t)\in W^{1,2}([0,T],\mathbb{C}^{n})\,\big|\, z(0)=Sz(T)\Big\}.
\end{eqnarray*}  $D$ is a bounded operator acting on $L^2(0,T;\dbR^n)$ defined by $(Dz)(t)=D(t)z(t)$.
In this paper, we will prove the following Hill-type formula.
\begin{thm}\label{thm1.1}Let $D\in \mathfrak{B}(n)$ and $S$ be an orthogonal matrix. Then, for any $\nu\in\mathbb{C}$
\begin{eqnarray}
\det\[\(\d-D+\nu I_n\)\(\d+\hat{P}_0\)^{-1}\]=(-1)^n|C({S})| e^{-\frac{n\nu T}{2}} e^{-\frac{1}{2}\int_0^T Tr(D)dt} \det(S\gamma_D(T)-e^{\nu T}I_n),\label{eq0.1a}
\end{eqnarray}
where  $\hat{P}_0$ is the orthogonal projection onto $\ker (S-I_n)$ and $C(S)$ is a constant depending only on $S$.
\end{thm}
Here, if we let $W=\ker(S-I_n)^\perp$ and $k_0=\dim\ker(S-I_n)$, then $C(S)=T^{-k_0}\frac{1}{\det[(S-I_n)|_{W}]}$.
\begin{rem}
\begin{itemize}
\item[1)]
In Theorem \ref{thm1.1}, since $(D+\hat{P}_0-\nu I_n)\(\d+\hat{P}_0\)^{-1}$ is not a trace class operator, but a Hilbert-Schmidt operator. Hence, the infinite determinant
\begin{eqnarray*}
\det\[\(\d-D+\nu I_n\)\(\d+\hat{P}_0\)^{-1}\]=\det\[id-(D+\hat{P}_0-\nu I_n)\(\d+\hat{P}_0\)^{-1}\]
\end{eqnarray*}
is not the classical Fredholm determinant. In fact, we can define the determinant in the following way. Let $\hat{P}_N$ be the orthogonal projections onto
\begin{eqnarray*}
V_N=\bigoplus\limits_{\nu\in\sigma(\d),|\nu|\leq N} \ker \(\nu-\d\).
\end{eqnarray*}
And the conditional Fredholm determinant
\begin{eqnarray*}
\det\[\(\d-D+\nu I_n\)\(\d+\hat{P}_0\)^{-1}\]=\lim\limits_{N\to\infty} \det\[ id-\hat{P}_N(D+\hat{P}_0-\nu I_n)\(\d+\hat{P}_0\)^{-1}\hat{ P}_N\].
\end{eqnarray*}
\item[2)] If $S=I_n$, then the boundary condition problem (\ref{eq0.12}-\ref{eq0.13}) is the canonical periodic boundary condition problem. In this case $C(S)=T^{-n}$.  When the period $T=1$, the Hill-type formula was obtained by Denk \cite{De}. Based on the Hill-tpye formula, Denk developed some efficient numerical method for the ODEs.
\end{itemize}
\end{rem}

Similar to \cite{HOW}, the Hill-type formula  (\ref{eq0.1a}) is the starting point of the Krein-type trace formula.
 To get the trace formula, for
$D_0(t), D(t)\in  \mathfrak{B}(n)$,  we consider the eigenvalue problem
\begin{eqnarray} \dot{z}(t)&=& (D_0(t)+\alpha D(t))z(t), \label{0.1.5}\\
 z(0)&=& S z(T), \label{0.1.6}\end{eqnarray}
that is, to find the $\alpha\in\mathbb C$ such that the system (\ref{0.1.5}-\ref{0.1.6}) has a nontrivial solution.
As above, let $\gamma_\alpha$ be the fundamental solution of (\ref{0.1.5}).  To state  the trace formula, we  need some notations.
 Write $M=S\gamma_0(T)$ and $\hat{D}(t)=\gamma_{0}^{-1}(t)D(t) \gamma_{0}(t)$. For $k\in\mathbb{N}$, let
\begin{eqnarray*}
M_k=\int_0^T\hat{D}(t_1)\int_0^{t_1}\hat{D}(t_2)\cdots\int_0^{t_{k-1}}\hat{D}(t_k)dt_k\cdots
dt_2 dt_1, \label{adc0.4.13}
\end{eqnarray*} and
\begin{eqnarray*}\label{adc4.14}
G_k=M_k M\left(M-e^{\nu T} I_{2n}\right)^{-1}.
\end{eqnarray*}

\begin{thm}\label{thm1.2}
Let $\nu\in \mathbb C$ such that $\d-{D}_0-\nu $ is invertible,  $F=D\left(\d-D_0+\nu I_n\right)^{-1}$,  then
\begin{eqnarray}
Tr(F)=\frac{1}{2}\int_0^T Tr ( D(t))dt-Tr(G_1)
\label{th1.2f1},
\end{eqnarray}
and for any positive integer $m\geq2$,
\begin{eqnarray}\label{0.0.0}
Tr(F^m)=m\sum_{k=1}^m
\frac{(-1)^{k}}{k}\Big[\sum\limits_{j_1+\cdots+j_k=m}Tr(G_{j_1}\cdots
G_{j_k})\Big]. \label{th1.2f2}\end{eqnarray}

\end{thm}
\begin{rem}
\begin{itemize}
 \item[(1).] For $m=1$, $F$ is not a trace class operator but a Hilbert-Schmidt operator. And hence $Tr(F)$ is not the usual trace but a kind of conditional trace\cite{HW1}.  That is
 \bea Tr(F) = \lim\limits_{N\to\infty} Tr  \hat{P}_N F \hat{P}_N.  \eea

 For $m\geq 2$, $F^m$ are trace class operators. Obviously, for $\nu=0$,  $\lambda_i$ is the eigenvalues of (\ref{0.1.5})-(\ref{0.1.6}), if and only if $1\over\lambda_j$ is an eigenvalue of $F$.
$$
\sum\limits_{j}{1\over \lambda_j^m}= m\sum_{k=1}^m
\frac{(-1)^{k}}{k}\Big[\sum\limits_{j_1+\cdots+j_k=m}Tr(G_{j_1}\cdots
G_{j_k})\Big],
$$
where the sum takes for eigenvalues  counting algebraic multiplicity.

\item[(2).]  The trace formula (\ref{th1.2f1}) for ODE is different from that for Hamiltonian system, and the reason is that the differential operator here is not self-adjoint any more.
\end{itemize}
\end{rem}

\begin{rem}
The idea and the techniques are similar to those in \cite{HOW}, however, there are at least two reasons to write this paper. For the first, it is not like that Hamiltonian system comes from mechanic system mostly, the ODE systems come from many areas, then the Hill-type trace formula and Krein-type trace formula for ODE will be more convenient to be used.  For the second, since the differential operators for ODEs are not self-adjoint, we should do some spectral analysis for ODE carefully. Although the formulas for ODE are similar to that for Hamiltonian system in \cite{HOW}, however, it is difficult to deduce them.
\end{rem}

The trace formula for Hamiltonian systems is a useful tool in study the stability of Hamiltonian systems, by using the trace formula and Maslov-type index theory \cite{Lon4},  some applications to the $n$-body problem is given in  \cite{HOW,HO}.
In Section 5, as an application, we will   give a generalization of Krein's work on second order systems \cite{Kr1}.

This paper is organized as follows, in section 2, we review  the basic properties of conditional Fredholm  determinant and conditional trace. In section 3, we derive the Hill-type formula for the $S$-periodic orbits in ODE. In section 4,
we get the trace formula from the Hill-type  formula.  Finally, as an example, we will reformulate Krein's trace formula from our viewpoint.

\section{Preliminaries}

In this section, we will mainly recall some fundamental properties of conditional Fredholm determinant, which was developed in \cite{HW1}. In the classical settings, the Fredholm determinant $\det(id+F)$ is defined for a trace class operator $F$, details could be found in \cite{Si}. However, when we study the ODEs, the operators we encountered are of the form $(id+F)$ with that $F$ is not a trace class operator, but a Hilbert-Schmidt operator. Thus, the Fredholm determinant $\det(id+F)$ could not be defined, and the conditional Fredholm determinant will be used instead. To define the conditional Fredholm determinant , the trace finite condition plays an important role.

Let $\{P_k\}$ be a sequence of finite rank projections, such that the following conditions are satisfied,
\begin{itemize}
\item[(1)] for $k\leq m$, $Range(P_k)\subseteq Range (P_m)$,
\item[(2)] $P_k$ converges to $id$ in the strong operator topology.
\end{itemize}
A Hilbert-Schmidt operator $F$ is called to have the trace finite condition with respect to $P_k$, if the limit $\lim\limits_{k\to\infty} Tr(P_k F P_k)$ exists, and the limit is finite. Clearly, a trace class operator has the trace finite condition.  Obviously, all the Hilbert-Schmidt operators with trace finite conditions consists a linear space, that is, if $F_1$ and $F_2$ are Hilbert-Schmidt operators with trace finite condition, then $\alpha_1 F_1+\alpha_2 F_2$ has the trace finite condition.

By \cite{Si}, if $F$ is a Hilbert-Schmidt operator, then the regularized Fredholm determinant  is defined by
\begin{eqnarray}
{\det}_2(id+F)=\det\((id+F)e^{-F}\).
\end{eqnarray}
As been pointed in \cite{HW1}, if $F$ is a Hilbert-Schmidt operator with trace finite condition, then the conditional Fredholm determinant can be defined by
\begin{eqnarray}\label{eq1.5a}
\det(id+F)&=&\lim\limits_{k\to\infty} \det(id+P_k F P_k)\nonumber\\&=&{\det}_2(id+F)\lim\limits_{k\to\infty}e^{ Tr(P_k F P_k)},
\end{eqnarray}
where, $P_k F P_k$ are finite rank operators, and hence $\det(id+P_k F P_k)$ is well defined. As we proved in \cite{HW1}, many fundamental properties of conditional Fredholm determinant are similar to that of the usual Fredholm determinant.

\begin{prop}\label{prop2.5}
\begin{itemize}
 \item[1)] If $F_1$ and $F_2$ are Hilbert-Schmidt operators with trace finite condition, then
\begin{eqnarray*}
\det\left((id+F_1)(id+F_2)\right)=\det(id+F_1)\det(id+F_2).
\end{eqnarray*}
\item[2)] Let $E=E_1\oplus E_2$, and  $F_i$ be Hilbert-Schmidt operators on $E_i$ with trace finite condition with respect to $P_k^{(i)}$, $i=1,2$. Let   $F=F_1\oplus F_2$, then $F$ has the trace finite condition with respect to $P_k^{(1)}\oplus P_{k}^{(2)}$, and
\begin{eqnarray*}
\det(id+F)=\det(id_{E_1}+F_1)\det(id_{E_2}+F_2),
\end{eqnarray*}
where $id_{E_i}$ are identities on $E_i$, for $i=1,2$.
\end{itemize}
\end{prop}

Similar to that we have given in \cite{HW1}, it is not hard to show that $\det(id+\alpha F)$ is analytic on $\alpha$, for a Hilbert-Schmidt operator $F$ with trace finite condition. For reader's convenience, we will give the proof of it.
\begin{lem}\label{lem2.1}
Let $F$ be a Hilbert-Schmidt operator with trace finite condition with respect to $\{P_k\}$,  then $\det(id+\alpha F)$ is an entire function.
\end{lem}
\begin{proof} Write
\begin{eqnarray}
f_k(\alpha)=\det(id+\alpha P_k F P_k),
\end{eqnarray}
then, by the definition, $f_k$ converges to $f(\alpha)=\det(id+\alpha P_k F P_k)$ point-wisely. Moreover, since $P_k$ are finite rank projections, $P_k F P_k$ are finite rank operators, and hence they are trace class. It follows that $f_k(\alpha)$ are entire functions.
By Montel's Theorem, it suffices to show that, $\{f_k\}$ is locally bounded, that is, for any compact set $K\subseteq \mathbb C$, there is a constant $C>0$ depending only on $K$ such that $$\sup\big\{ |f_k(\alpha)|\,\big|\,\alpha\in K\big\}<C.$$ By \cite[Lemma 2.3]{HW1},
$$
f_k(\alpha)={\det}_2 (id+\alpha P_kFP_k) e^{-\alpha Tr P_k F P_k}.
$$
Firstly, by \cite[Theorem 9.2]{Si},
\begin{eqnarray*}
\sup\limits_{\alpha\in K}|{\det}_2 (id+\alpha P_kFP_k)|\leq \sup\limits_{\alpha\in K}e^{C||\alpha P_kFP_k||_2^2}\leq e^{C||F||_2^2\sup\limits_{\alpha\in K}|\alpha|},
\end{eqnarray*}
 where $||\cdot||_2$ is the Hilbert-Schmidt norm. Secondly, since $\lim\limits_{N\to\infty}Tr P_N F P_N$ exists, we have that,  there is a constant $C_1>0$ such that for any $k\in\mathbb N$, $|Tr P_k F P_k|<C_1$. It follows that
 $$\sup\limits_{\alpha\in K}|e^{-\alpha Tr P_k F P_k}|\leq e^{C_1\sup\limits_{\alpha\in K}|\alpha|}.$$ Therefore,  $$\sup\big\{|f_k(\alpha)|\,\big|\,\alpha\in K\big\}<C$$ for some constant $C$ which depends only on $K$, that is, $\{f_k\}$ is locally bounded. The proof is complete.
\end{proof}

\begin{rem}
In the proof of Lemma \ref{lem2.1}, we show that $\{\det(id+\alpha P_k F P_k)\}$ is a normal family, and hence there is a subsequence of $\{\det(id+\alpha P_k F P_k)\}$, say $\{\det(id+\alpha P_{k_j}F P_{k_j})\}$, which is convergent uniformly to $\det(id+\alpha F)$ on any compact subset of $\mathbb C$. Denote by $$g_{k_j}(\alpha)=\det(id+\alpha P_{k_j}F P_{k_j}).$$
\end{rem}

Following \cite{Si}, we call a
function $f:X\rightarrow Y$ between Banach spaces, finitely analytic
if and only if, for all $A_1,...,A_n\in X$, $f(z_1A_1+...+z_nA_n)$
is an entire function of $z_1,...,z_n$ from $\mathbb{C}^n$ to $Y$.
This concept is very useful in studying  the conditional Fredholm determinant. An important property is the following:
\begin{thm}\label{th2.b0} \cite[pp.45, Theorem5.1]{Si}.  If a finitely analytic function $f$ satisfied $f(x)\leq G(\parallel x\parallel)$ for some monotone function $G$ on $[0,\infty)$, then
 $f$ is Fr\'{e}chet differentiable for all $x\in X$, and $Df$ is finitely analytic function from $X$ to $\mathfrak{L}(X,Y)$(Banach space of the linear operators from $X$ to $Y$), and  $$ (Df)(x)\leq G(\parallel x\parallel+1).$$
\end{thm}

Next, we will consider the Taylor expansion of $\det(id+\alpha F)$. Write
\begin{eqnarray}
g(\alpha)=\det(id+\alpha F).
\end{eqnarray}
 By \cite[Theorem 5.4]{Si}, we have the following lemma.
\begin{lem}
Let $g_k(\alpha)=\det (id+\alpha F_k)$. Then  the Taylor
expansion near $0$ for $g_k(\alpha)$ is
\begin{eqnarray*}
g_k(\alpha)=\sum\limits_{m=0}^\infty\alpha^m a_{k,m}/m! ,
\end{eqnarray*}
where
\begin{eqnarray*}
a_{k,m}=\det\left( \begin{array}{ccccc}TrF_k& m-1&0 &\cdots & 0 \\ Tr(F_k^2)& TrF_k & m-2 &\cdots & 0\\ \vdots & \vdots &\ddots& \ddots & \vdots \\
Tr(F_k^{m-1}) & Tr(F_k^{m-2}) & \cdots & TrF_k & 1\\
 Tr(F_k^m) & Tr(F_k^{m-1})& \cdots & Tr(F_k^2) & TrF_k   \end{array}\right).
\end{eqnarray*}
and $F_k=P_k F P_k$.
\end{lem}
The following reasoning is similar to that in \cite[Section 2]{HOW},  and we write here again for reader's convenience.

Now, let  $g(\alpha)=\sum\limits_{m} {a_m\over m!} \alpha^m$ be the Taylor expansion of $g(\alpha)$.
Since $g_{k_j}(\alpha)$ converges to $g(\alpha)$ on any compact subset of $\mathbb C$, the coefficients  $a_{k_j,m}\to a_m$. Notice that $F$ is a Hilbert-Schmidt operator with trace finite condition, then $Tr F_k \to Tr F$. Therefore
\begin{eqnarray*}
a_m=\det\left( \begin{array}{cccc}TrF& m-1& \cdots & 0 \\
Tr(F^2)& TrF & \cdots & 0\\ \vdots & \vdots & \ddots & \vdots
\\ Tr(F^m) & Tr(F^{m-1})& \cdots & TrF
\end{array}\right),
\end{eqnarray*}

Note that for $\alpha$ small, by \cite[p.47,(5.12)]{Si}, we have that
\begin{eqnarray}
\det(id+\alpha F_k)=\exp\(\sum\limits_{m=1}\frac{(-1)^{m+1}}{m}\alpha^m Tr(F_k^m)\).
\end{eqnarray}
By taking limit, we have the following theorem.
\begin{thm}\label{thm2.4}
Let $g(\alpha)=\det(id+\alpha F)$. Then for $\alpha$ small,
\begin{eqnarray}
g(\alpha)=\exp\(\sum\limits_{m=1}^\infty\frac{(-1)^{m+1}}{m}\alpha^m Tr(F^m)\).
\end{eqnarray}
\end{thm}

At last of this section, we will brief  review  the Hill-type formula for Hamiltonian systems.  Assume $\bar{S}$ be a symplectic orthogonal matrix,    for $\bar{S}$-periodic orbits in Hamiltonian system, the Hill-type formula was proved in \cite{HW1}, which is listed as follows.

\begin{thm}\label{th2.b1}There is a constant $C(\bar{S})>0$, which depends only on $\bar{S}$,  such that for any $\nu\in\mathbb{C}$
 \bea  \det\(\(-J\d-B-\nu J\)\(-J\d+P_0\)^{-1}\)=C(\bar{S})e^{-\frac{1}{2}\int_0^T Tr (JB(t)) dt}\lambda^{-n}\det(\bar{S}\gamma(T)- \lambda I_{2n}), \label{thf2.b1} \eea
 where $\lambda=e^{\nu T}$ and $C(\bar{S})>0$ is a constant depending only on $\bar{S}$.
\end{thm}
Precisely, \bea C(\bar{S})=2^{-n}(2/T^2)^{k_0}\prod_{j=k_0+1}^n\frac{1}{1-\cosh(\sqrt{-1}\nu_j T)}, \label{cs}\nonumber \eea where $k_0=\dim\ker(R-\sqrt{-1}Q-I_n)$.
 In fact, if we denote $\bar{W}=\ker(\bar{S}-I_{2n})^\bot$, then
$$C(\bar{S})=T^{-2k_0}\frac{1}{\det(\bar{S}-I_{2n})|_{\bar{W}}}. $$
 Obviously, $C(\bar{S})>0$. If $\bar{S}=I_{2n}$, i.e. for the periodic boundary conditions, $C(\bar{S})=T^{-2n}$, and if $\ker(\bar{S}-I_{2n})=0$,
then \bea C(\bar{S})=\det(\bar{S}-I_{2n})^{-1}. \label{csn}\nonumber\eea
For a $n\times n$ orthogonal matrix  $S$, let $W=\ker(S-I_n)^\bot$, and we also denote
\bea C(S)=T^{-k_0}\frac{1}{\det(S-I_n)|_W}, \label{csode} \nonumber\eea where $k_0=\dim\ker(S-I_n)$.
It can be shown that $C(S)=(-1)^{\sigma}|C(S)|$, where $\sigma$ is the orientation of $S$.
 Then
in the special case  $\bar{S}=\left(\begin{array}{cc}S & 0_n\\ 0_n & S\end{array}\right)$, where $S$ is a $n\times n$ orthogonal matrix,
 \bea C(\bar{S})=C(S)^2. \label{csn2}\eea

\section{Hill-type formula for $S$-periodic solutions of ODEs}
In this section, we will mainly derive the Hill-type formula for the first order ODE.
\subsection{Hill-type formula for the first order ODEs}
To derive the Hill-type formula, we will consider the conditional Fredholm determinant of ODEs. Firstly, we consider the first order ODE with the $S$-periodic  boundary condition,
\begin{eqnarray}
\dot{z}(t)&=&D(t) z(t)\\
      z(0)&=& S z(T),
\end{eqnarray}
where $S$ is an orthogonal matrix on $\mathbb R^n$ and $D\in\mathfrak{B}(n)$.  To continue, we will consider the spectrum of $\d$ with the $S$-boundary condition. Denote the domain of $\d$ by
$$
D_{S}=\Big\{z(t)\in W^{1,2}([0,T],\mathbb{C}^{n})\,\big|\, z(0)=Sz(T)\Big\}.
$$
Obviously $\sqrt{-1}\d$ is a self-adjoint operator on $E=L^2([0,T],\mathbb{C}^{2n})$ with domain $D_S$.

For simplicity, write $A=\sqrt{-1}\d$.
By some simple calculation, we have that $\nu\in\sigma(A)$ if and only if $\ker (S e^{\sqrt{-1}\nu I_n T}-I_n)$ is nontrivial, and the corresponding eigenvector is $e^{\nu \sqrt{-1}I_n t}\xi$, where $\xi \in \ker (S e^{\sqrt{-1}\nu I_n T}-I_n)$. Since $S$ is an orthogonal  matrix, there is a unitary matrix $U$ such that
\begin{eqnarray*}
U^* S U=\left(\begin{array}{cccc}e^{-\sqrt{-1}\theta_1} & & & \\ & e^{-\sqrt{-1}\theta_2} & & \\ & & \ddots & \\ & & & e^{-\sqrt{-1}\theta_n} \end{array}\right).
\end{eqnarray*}
where $0= \theta_1 =\cdots= \theta_{k_0}<\theta_{k_0+1} \leq\cdots \leq \theta_n<2\pi$. It is easy to check that $\nu \in \sigma(A)$ if and only if $\nu=(\theta_j+2k\pi)/T$, for some $1\leq j\leq n$; furthermore, $(U^* S U e^{\sqrt{-1}\nu_j I_n T}-I_n)\xi_j=0$ if and only if $\xi_j=(\underbrace{0,0,\cdots,0}_{j-1},1,0,\cdots,0).$
We have the following lemma.
\begin{lem}\label{lem3.3a}
The spectrum of $\d$ (with the $S$-boundary condition) is periodic with the period of $2\pi/T$. Precisely, let $\nu_j=\theta_j/T$, then
\begin{eqnarray*}
\sigma\(\d\)=\bigcup\limits_{j=1}^n\big\{-\sqrt{-1}(\nu_j+2k\pi/T)\,\big|\,k\in\mathbb Z\big\}.
\end{eqnarray*}
In fact $\nu_j=\theta_j/T$.
Moreover, let $-\sqrt{-1}(\nu_j+2k\pi/T)\in\sigma(\d)$, then the corresponding eigenvector is $e^{(\nu_j+2k\pi/T)\sqrt{-1}I_nt}U\xi_j$, where $\xi_j=(\underbrace{0,0,\cdots,0}_{j-1},1,0,\cdots,0)$.
\end{lem}
It is worth being pointed out that
\begin{eqnarray}
0=\nu_1=\cdots=\nu_{k_0}<\nu_{k_0+1}\leq\cdots\leq \nu_{n}<2\pi/T,
\end{eqnarray}
where $k_0=\dim\ker(S-I_n)$.
\begin{rem}
Since $S$ is an orthogonal matrix, $e^{i\theta}$ is an eigenvalue of $S$ if and only if $e^{-i\theta}$ is an eigenvalue of $S$. By the argument before Lemma \ref{lem3.3a} we know that, under the $S$-boundary condition, the spectrum \begin{eqnarray}\sigma\(\d\)=\sigma\(-\d\),\end{eqnarray} which will be used later.
\end{rem}

By Lemma \ref{lem3.3a},  $\left(\d+\hat{P}_0\right)^{-1}$ is a Hilbert-Schmidt operator, where $\hat{P}_0$ is the orthogonal projections onto $\ker\left(\d\right)$.
Applying a similar reasoning to \cite[Remark 2.9]{HW1} shows that  $(D+\hat{P}_0)(\d+\hat{P}_0)^{-1}$ has the trace finite condition with respect to $\{\hat{P}_N\}$, where $\hat{P}_N$ are orthogonal projections onto
\begin{eqnarray}
V_N=\bigoplus\limits_{\nu\in\sigma(\d),|\nu|\leq N} \ker \left(\nu-\d\right).
\end{eqnarray}
In fact, we have the following lemma. For reader's convenience, we will give the proof of it.
\begin{lem}\label{lem3.3b}Under the above assumption, we have that
\begin{eqnarray*}
\lim\limits_{N\rightarrow\infty}Tr\(\hat{P}_ND\(\d+\hat{P_0}\)^{-1}\hat{P}_N\)=\frac{1}{T}\sum\limits_{j=1}^{k_0}\int_0^T \hat{D}_{jj} dt+\sqrt{-1}\sum\limits_{j=k_0+1}^n \frac{1+\cos T \nu_j}{2\sin T\nu _j}\int_0^T \hat{D}_{jj}(t) dt,
\end{eqnarray*}
where $\hat{D}=U^\ast D U$.
\end{lem}
\begin{proof}
For $1\leq j\leq n$, let
\begin{eqnarray}
M_j=\bigoplus\limits_{k\in\mathbb Z} \ker \(\d+\sqrt{-1}\(\nu_j+{2k\pi\over T}\)\),
\end{eqnarray}
then $E=\oplus M_j$. Let $\frac{1}{\sqrt{T}}e^{(\nu_j+2k\pi/T)\sqrt{-1}I_nt}U\xi_j$ be the eigenvector for $-\d$ with respect to the eigenvalue $-\sqrt{-1}\(\nu_j+{2k\pi\over T}\)$. Then
\begin{eqnarray*}
&&\lim\limits_{N\to\infty}Tr\(\hat{P}_ND\(\d+\hat{P_0}\)^{-1}\hat{P}_N\)\\
&&\ \ =\sum\limits_{j=1}^n \lim\limits_{N\to\infty}\sum\limits_{|k|<N}\left\langle D\(\d+\hat{P_0}\)^{-1} \frac{1}{\sqrt{T}}e^{(\nu_j+2k\pi/T)\sqrt{-1}I_nt}U\xi_j,\frac{1}{\sqrt{T}}e^{(\nu_j+2k\pi/T)\sqrt{-1}I_nt}U\xi_j\right\rangle.
\end{eqnarray*}
For $1\leq j\leq k_0$,
$$
 \lim\limits_{N\to\infty}\sum\limits_{|k|<N}\left\langle D\(\d+\hat{P_0}\)^{-1} \frac{1}{\sqrt{T}}e^{(\nu_j+2k\pi/T)\sqrt{-1}I_nt}U\xi_j,\frac{1}{\sqrt{T}}e^{(\nu_j+2k\pi/T)\sqrt{-1}I_nt}U\xi_j\right\rangle=\frac{1}{T}\int_0^T \hat{D}_{jj}(t)dt.
$$
For $k_0+1\leq j\leq n$,
\begin{eqnarray*}
&&\lim\limits_{N\to\infty}\sum\limits_{|k|<N}\left\langle D\(\d+\hat{P_0}\)^{-1} \frac{1}{\sqrt{T}}e^{(\nu_j+2k\pi/T)\sqrt{-1}I_nt}U\xi_j,\frac{1}{\sqrt{T}}e^{(\nu_j+2k\pi/T)\sqrt{-1}I_nt}U\xi_j\right\rangle\\
&&\ =2\lim\limits_{N\to\infty}\sum\limits_{k\leq N}{-{1\over \sqrt{-1}(\nu_j+2k\pi/T)}}\int_0^T \hat{D}_{jj}(t) dt\\
&&\ ={\sqrt{-1}}\frac{1+\cos T\nu_j}{\sin T\nu_j}\int_0^T \hat{D}_{jj}(t)dt.
\end{eqnarray*}
The above calculations imply the desired result.
\end{proof}
Now, we will consider the conditional Fredholm determinant $\det\left((\d-D)(\d+\hat{P}_0)^{-1}\right)$ with respect to $\hat{P}_N$. As we have done before, write
\begin{eqnarray*}
\(\d-D\)\(\d+\hat{P}_0\)^{-1}=id-(D+\hat{P}_0)\(\d+\hat{P}_0\)^{-1}.
\end{eqnarray*}

Hence, the conditional Fredholm determinant $\det\left(id-(D+\hat{P}_0)(\d+\hat{P}_0)^{-1}\right)$ is well-defined.
Please note that for any $D_1\in\mathfrak{B}(n)$ such that $\d+D_1$ is invertible, the operator $D(\d+D_1)^{-1}$ is a Hilbert-Schmidt operator with the trace finite condition with respect to $\hat{P}_N$. Therefore the infinite determinant $\det\left[\(\d-D\)\(\d+D_1\)^{-1}\right]$ is well defined.

In this remaining part of this subsection, we will deduce the Hill-type formula for first order ODE with $S$-periodic boundary condition, where $S$ is an orthogonal  matrix. Consider the following equation
\begin{eqnarray*}
 \dot{u}(t)&=&D(t)u(t),\label{eq3.1}\\
   u(0)&=&Su(T)\label{eq3.2}.
\end{eqnarray*}
To obtain the Hill-type formula of the above system, we will consider the following Hamiltonian system,
\begin{eqnarray}
\dot{z}(t)&=&JB(t)z(t) \label{3.b1}, \\
z(0)&=&\bar{S}z(T) \label{3.b2},
\end{eqnarray}
where $B(t)=V\left(\begin{array}{cc}iD & \\ & 0_n\end{array}\right)V^{-1}$, $\bar{S}=\left(\begin{array}{cc} S & \\ & S\end{array}\right)$.
Changing  the basis by $V$, we have
\begin{eqnarray*}
V^*\(-J\d\)V=\left(\begin{array}{cc}i\d & \\ & -i\d\end{array}\right),\   V^*B(t)V=\left(\begin{array}{cc}iD & \\ & 0_n\end{array}\right)\ \text{and}\ V^*JV=\left(\begin{array}{cc}-i I_n &
\\ & i I_n\end{array}\right).
\end{eqnarray*}
Let $\gamma(t)$ be the fundamental solution of (\ref{3.b1}). Under the new basis, it is obvious that $$ \gamma(t)=\left(\begin{array}{cc}\gamma_D(t) & \\ & I_n\end{array}\right), $$
where $\gamma_D(t)$ satisfied $\dot{\gamma}_D(t)=D(t)\gamma_D(t)$ with $\gamma_D(0)=I_n$.

By the Hill-type formula (\ref{thf2.b1}) for Hamiltonian system, we have
\begin{eqnarray*}
&&\det\left[ \left(\left(\begin{array}{cc}i\d & \\ & -i\d\end{array}\right)-\left(\begin{array}{cc}iD & \\ & 0_n\end{array}\right)-\left(\begin{array}{cc}-i\nu I_n &
\\ & i\nu I_n\end{array}\right)\right)\left(\begin{array}{cc}i\d +P_0 & \\ & -i\d+P_0\end{array}\right)^{-1}\right]\\
&&\ \ =C(\bar{S}) e^{-n\nu T}\exp\left[-{\frac{1}{2}\int_0^T Tr \left(\begin{array}{cc}  D & \\ & 0_n\end{array}\right)}dt\right]\det \left(\bar{S}\gamma(T)-e^{\nu T}I_{2n}\right).
\end{eqnarray*}
By Proposition \ref{prop2.5}(2), we can  rewrite the above equation as
\begin{eqnarray}\label{eq3.5}
&&\det\left[\left(i\d-iD+i\nu I_n\right)\left(i\d+\hat{P}_0\right)^{-1}\right]\cdot\det\left[\left(-i\d-i\nu I_n\right)\left(-i\d+\hat{P}_0\right)^{-1}\right]\nonumber\\&&=C(\bar{S})e^{-n\nu T}\exp\left[-\frac{1}{2}\int_0^T Tr(D) dt\right]\det\left(S\gamma_D(T)-e^{\nu T}I_n\right)\det\left(S-e^{\nu T}I_n\right).
\end{eqnarray}
To calculate the determinant  $\det\left[\left(i\d-iD+i\nu I_n\right)\left(i\d+\hat{P}_0\right)^{-1}\right]$, it suffices to calculate the determinant $\det\left[\left(-i\d-i\nu I_n\right)\left(-i\d+\hat{P}_0\right)^{-1}\right]$.
 Let $S_1=S|_W$, which is an orthogonal matrix on $\mathbb{R}^{n-k_0}$, obviously $S=\left(\begin{array}{cc}I_{k_0} & \\ & S_1\end{array}\right)$. We only need to compute $\det\left[\left(-i\d-i\nu I_{k_0}\right)\left(-i\d+\hat{P}_0\right)^{-1}\right]$ with $T$-periodic boundary condition  and  $\det\left[\left(-i\d-i\nu I_{n-k_0}\right)\left(-i\d\right)^{-1}\right]$ with  $S_1$ boundary  condition respectively. Firstly, notice that the spectrum of $i\d$ with $T$-periodic boundary condition
 $$\sigma\(i\d\)=\left\{\frac{2k\pi}{T}, k\in\mathbb{Z}\right\}.$$
 Moreover, by direct computation
 \bea \prod_{k\in\mathbb{Z}}(1+\frac{i\nu}{2k\pi/T}) &=& i\nu \prod_{k\in\mathbb{N}}(1+\frac{\nu^2}{(2k\pi/T)^2})  \nonumber \\ &=&\frac{2i}{T}\sinh(T\nu/2)
   \nonumber \\ &=& \frac{i}{T}e^{-T\nu/2}(e^{T\nu}-1). \label{ode.2}\nonumber  \eea
We have
\bea \det\left[\left(-i\d-i\nu I_{k_0}\right)\left(-i\d+\hat{P}_0\right)^{-1}\right]=(-1)^{k_0}i^{k_0}\frac{1}{T^{k_0}}e^{-Tk_0\nu/2}\det(I_{k_0}-e^{T\nu}). \label{ode3}\eea
Secondly, writing $\bar{S_1}=\left(\begin{array}{cc}S_1 & \\ & S_1\end{array}\right)$, by the Hill-type formula for Hamiltonian system with $\bar{S}_1$ boundary condition, we have
\begin{eqnarray*}
\det\left[\left(\begin{array}{cc}i\d+i\nu & \\ & -i\d-i\nu\end{array}\right)\left(\begin{array}{cc}i\d & \\ & -i\d\end{array}\right)^{-1}\right]=C(\bar{S}_1) e^{-(n-k_0)\nu T} \det\left(\bar{S}_1-e^{\nu T}I_{n-k_0}\right)^2.
\end{eqnarray*}
Therefore
\begin{eqnarray}
\det\left[\left(-i\d-i\nu I_{n-k_0}\right)\left(-i\d\right)^{-1}\right]=\sigma C(\bar{S}_1)^{1\over 2}e^{-{(n-k_0)\nu T\over 2}} \det\left(S_1-e^{\nu T}I_{n-k_0}\right), \label{ode5a}
\end{eqnarray}
where $\sigma$ can be $\pm1$, to be decided. Please note that, when we take $\nu=0$, the left hand side  of (\ref{ode5a}) equals to $1$ and the right hand side  equals to $\sigma \det\left(S_1-I_{n-k_0}\right)C(\bar{S})^{1\over 2}$. Since $C(\bar{S})>0$ and $\sigma$ is a square root of $1$, we have
$$\sigma=sign(\det\left(S_1-I_{n-k_0}\right)).$$
Since $C(\bar{S}_1)=\left[\det((S_1-I_{n-k_0})^{-1})\right]^2$, it follows that $\sigma C(\bar{S}_1)^{1\over 2}=\det((S_1-I_{n-k_0})^{-1})$, and hence
\begin{eqnarray}
\det\left[\left(-i\d-i\nu I_{n-k_0}\right)\left(-i\d\right)^{-1}\right]= C(S_1)e^{-{(n-k_0)\nu T\over 2}} \det\left(S-e^{\nu T}I_{n-k_0}\right), \label{ode5}
\end{eqnarray}
where $C(S_1)=\det((S_1-I_{n-k_0})^{-1})$. From (\ref{ode3}) and (\ref{ode5}), we have
 the following lemma.
\begin{lem}
Let $S$ be an orthogonal matrix. Then
\begin{eqnarray}\label{eq3.6}
 \det\left[\left(-i\d-i\nu I_n\right)\left(-i\d+\hat{P}_0\right)^{-1}\right]=(-1)^{k_0}i^{k_0}C(S)e^{-{n\nu T\over 2}} \det\left(S-e^{\nu T}I_{n}\right),
\end{eqnarray}
where $C(S)=\frac{C(S_1)}{T^{k_0}}$.
\end{lem}
Substituting (\ref{eq3.6}) in (\ref{eq3.5}), we have the following proposition.
\begin{prop} Under the above condition, we have
\begin{eqnarray*}
\det\left[\left(i\d-iD+i\nu I_n\right)\left(i\d+\hat{P}_0\right)^{-1}\right]=(-1)^{k_0}i^{k_0}C(S) e^{-\frac{n\nu T}{2}} e^{-\frac{1}{2}\int_0^T Tr(D)dt} \det(S\gamma_D(T)-e^{\nu T}I_n).
\end{eqnarray*}
\end{prop}
Notice that
\begin{eqnarray*}
&&\det\left[\left(i\d-iD+i\nu I_n\right)\left(i\d+\hat{P}_0\right)^{-1}\right]\\ &&\ \ =\det\left[\left(\d-D+\nu I_n\right)\left(\d-i\hat{P}_0\right)^{-1}\right]\\
&&\ \ =\det\left[\left(\d-D+\nu I_n\right)\left(\d+\hat{P}_0\right)^{-1}\right]\det\left[\left(\d+\hat{P}_0\right)\left(\d-i\hat{P_0}\right)^{-1}\right]\\
&&\ \ =\det\left[\left(\d-D+\nu I_n\right)\left(\d+\hat{P}_0\right)^{-1}\right]i^{k_0},
\end{eqnarray*}
where the second equality holds true because of Proposition \ref{prop2.5}(1).
Therefore, we have the following theorem which is just Theorem \ref{thm1.1}.
\begin{thm}\label{thm3.6}Under the above assumption,
\begin{eqnarray}
&&\det\left[\left(\d-D+\nu I_n\right)\left(\d+\hat{P}_0\right)^{-1}\right]\nonumber\\&&=(-1)^{k_0}C(S) e^{-\frac{n\nu T}{2}} e^{-\frac{1}{2}\int_0^T Tr(D)dt} \det(S\gamma_D(T)-e^{\nu T}I_n). \label{thmf}
\end{eqnarray}
\end{thm}
Please note that $$(-1)^{k_0}C(S)=(-1)^n|C(S)|, $$
thus we have
\begin{eqnarray*}
&&\det\left[\left(\d-D+\nu I_n\right)\left(\d+\hat{P}_0\right)^{-1}\right]\nonumber\\&&=(-1)^{n}|C(S)| e^{-\frac{n\nu T}{2}} e^{-\frac{1}{2}\int_0^T Tr(D)dt} \det(S\gamma_D(T)-e^{\nu T}I_n). \label{thmf1}
\end{eqnarray*}

\begin{rem}\label{fi.1}
From \cite{HW1},  $\left|\det\(\(-J\d-B\)\(-J\d+P_0\)^{-1}\)\right|\leq G(\parallel B\parallel)$ for some monotone function $G$. By the above analysis, we have
\bea \det\left[\left(\d-D+\nu I_n\right)\left(\d+\hat{P}_0\right)^{-1}\right]=\det\(\(-J\d-B\)\(-J\d+P_0\)^{-1}\)\cdot d(\nu),   \nonumber \eea
where $d(\nu)$ is a function depending only on $\nu$ and  $B(t)=V\left(\begin{array}{cc}iD & \\ & 0_n\end{array}\right)V^{-1}$. Thus we have that
$ \det\left[\left(\d-\cdot\right)\left(\d+\hat{P}_0\right)^{-1}\right]$ is a finitely analytic function from $\mathfrak{B}(n)$ to $\mathbb{C}$, and satisfied
\bea \left|\det\left[\left(\d-D\right)\left(\d+\hat{P}_0\right)^{-1}\right]\right|\leq G_1(\parallel D\parallel),\label{3c.1} \nonumber\eea
for some monotone function $G_1$. Moreover, from Theorem \ref{th2.b0}.  suppose
$\Omega\subseteq \mathbb{C}^m$ be an open set, and  $D(Z)$ is analytic map
from $\Omega\rightarrow \mathfrak{B}(n)$, then $\det\(\(\d-D(Z)\)\(\d+\hat{P}_0\)^{-1}\)$ is an analytic function on $\Omega$.
\end{rem}

By the multiplicative property of conditional Fredholm determinant,  Proposition 2.1(1) , we have the following corollary.
\begin{cor}
Let $D\in\mathfrak{B}(n)$ such that $\d-D$ is invertible, then
\begin{eqnarray}
&&\det\[\(\d-D_1+\nu I_n\)\(\d-D\)^{-1}\]\nonumber\\&&=e^{-n\nu T/2}e^{-{1\over 2}\int_0^T Tr(D-D_1)dt}\det\left(S\gamma_D(T)-e^{\nu T}I_n\right)\det\left(S\gamma_{D_1}-I_n\right)^{-1}.  \label{c1}
\end{eqnarray}
\end{cor}

\section{Trace formula for 1st order ODE}

In this section, we will derive the trace formula from the Hill-type formula. Firstly, we will study the Taylor expansion for linear parameterized Monodromy matrices, which is similar to the case of Hamiltonian systems \cite{HOW}.

\subsection{Taylor expansion for linearly parameterized Monodromy matrices}
Let $D_0, D\in\mathfrak{B}(n)$.
For $\alpha \in \mathbb C$, set $D_\alpha=D_0+\alpha D$, for $\alpha\in\mathbb C$, let $\gamma_\alpha$ be the corresponding fundamental solutions, that is
\begin{eqnarray*}
\dot{\gamma}_\alpha(t)=D_\alpha(t)\gamma_\alpha(t) \label{c4.1}.
\end{eqnarray*}
Fixed $\alpha_0\in\mathbb C$, direct computation  shows that
\begin{eqnarray*}
\frac{d}{dt}(\gamma_{\alpha_0}^{-1}(t)\gamma_\alpha(t))&=& \gamma_{\alpha_0}^{-1}(t)(D_\alpha(t)-D_{\alpha_0}(t))\gamma_\alpha(t) \nonumber \\
&=& (\alpha-\alpha_0)\gamma_{\alpha_0}^{-1}(t)D(t)
\gamma_{\alpha_0}(t)\gamma_{\alpha_0}^{-1}(t)\gamma_\alpha(t)
 \label{c4.2}.
\end{eqnarray*}
Without loss of generality, assume $\alpha_0=0$. In what follows, write
\begin{eqnarray*}\hat{\gamma}_\alpha(t)=\gamma_{0}^{-1}(t)\gamma_\alpha(t),\end{eqnarray*} and \begin{eqnarray*}\hat{D}(t)=\gamma_{0}^{-1}(t)D(t) \gamma_{0}(t),\end{eqnarray*} thus
\begin{eqnarray}
\frac{d}{dt}\hat{\gamma}_\alpha(t)= \alpha
\hat{D}(t)\hat{\gamma}_\alpha(t) \label{c4.3}.
\end{eqnarray}
To simplify the notation, we use ``$^{(k)}$'' to denote the $k$-th
derivative on $\alpha$. Taking derivative on $\alpha$ for both sides
of (\ref{c4.3}), we get
\begin{eqnarray}
\frac{d}{dt}\hat{\gamma}^{(1)}_\alpha(t)&=&\hat{D}(t)\hat{\gamma}_\alpha(t)+\alpha
\hat{D}_\alpha(t)\hat{\gamma}^{(1)}_\alpha(t)\label{c4.4}.
\end{eqnarray}
By taking $\alpha=0$,  $\widehat{\gamma}_{0}(t)\equiv I_{n}$, we have
\begin{eqnarray*}
\hat{\gamma}^{(1)}_{0}(t)=\int_0^t\hat{D}(s)ds. \label{c4.5}
\end{eqnarray*}
Now, taking derivative on $\alpha$ for both sides of (\ref{c4.4}), we get
\begin{eqnarray*}
\frac{d}{dt}\hat{\gamma}^{(2)}_\alpha(t)&=&2\hat{D}(t)\hat{\gamma}^{(1)}_\alpha(t)
+\alpha\hat{D}(t)\hat{\gamma}^{(2)}_\alpha(t). \label{c4.6}
\end{eqnarray*}
Take $\alpha=0$, and we get
\begin{eqnarray*}
\hat{\gamma}^{(2)}_{0}(t)=2
\int_0^t\hat{D}(s)\hat{\gamma}^{(1)}_{0}(s) ds. \label{c4.7}
\end{eqnarray*}
By induction,
\begin{eqnarray*}
\frac{d}{dt}\hat{\gamma}^{(k)}_0(t)=k
 \hat{D}(t)\hat{\gamma}^{(k-1)}_0(t), \label{c4.9}
\end{eqnarray*}
and
\begin{eqnarray*}
\hat{\gamma}^{(k)}_0(t)=k
\int_0^t\hat{D}(s)\hat{\gamma}^{(k-1)}_0(s)ds. \label{c4.10}
\end{eqnarray*}
For $t=T$,  by  Taylor's formula,
\begin{eqnarray*}
\hat{\gamma}_{\alpha}(T)=I_{2n}+\alpha\hat{\gamma}^{(1)}_{0}(T)+\cdots+\alpha^k\hat{\gamma}^{(k)}_{0}(T)/k!
+\cdots, \label{c4.11}
\end{eqnarray*}
where
\begin{eqnarray*}\hat{\gamma}^{(1)}_{0}(T)=\int_0^T\hat{D}(t)dt
\end{eqnarray*}and
\begin{eqnarray*}
\hat{\gamma}^{(k)}_{0}(T)/k!=\int_0^T\hat{D}(t)\hat{\gamma}^{(k-1)}_{0}(t)/(k-1)!dt,
k\in\mathbb N  \label{c4.13}.  \end{eqnarray*} By induction, we have
\begin{eqnarray*}
\hat{\gamma}^{(k)}_{0}(T)/k!=\int_0^T\hat{D}(t_1)\int_0^{t_1}\hat{D}(t_2)\cdots\int_0^{t_{k-1}}\hat{D}(t_k)dt_k\cdots dt_2dt_1,
k\in\mathbb N\label{adc4.13}.  \end{eqnarray*}
 Obviously
$\hat{\gamma}_{\alpha}(T)$ is an entire function on the variable $\alpha$. We summarize the above reasoning as the following proposition.
\begin{prop}
Let $D_\alpha=D_0+\alpha D$, $\gamma_\alpha(T)$ be the corresponding fundamental solutions. Write $\hat{\gamma}_\alpha=\gamma_0^{-1} \gamma_\alpha$. Then, the Taylor expansion for $\hat{\gamma}_\alpha(T)$ at $0$ is
\begin{eqnarray}
\hat{\gamma}_{\alpha}(T)=I_{n}+\alpha\hat{\gamma}^{(1)}_{0}(T)+\cdots+\alpha^k\hat{\gamma}^{(k)}_{0}(T)/k!
+\cdots, \label{c4.11a}
\end{eqnarray}
where
\begin{eqnarray}
\hat{\gamma}^{(k)}_{0}(T)/k!=\int_0^T\hat{D}(t_1)\int_0^{t_1}\hat{D}(t_2)\cdots\int_0^{t_{k-1}}\hat{D}(t_k)dt_k\cdots dt_2dt_1,
k\in\mathbb N\label{adc4.13a}.  \end{eqnarray}
\end{prop}

In what follows, to simplify the notation,  set \begin{eqnarray*}
M(\alpha)=\hat{\gamma}_{\alpha}(T),\quad
M_0=I_{n} \quad \text{and}\quad M_j=\hat{\gamma}^{(j)}_{0}(T)/j!,\,j\in\mathbb N,\end{eqnarray*}  then
$$M(\alpha)=\sum_{j=0}^\infty \alpha^jM_j.$$

Set $M=S\gamma_0(T)$, then $S\gamma_\alpha(T)=M M(\alpha)$. For
$\lambda\in\mathbb C$, which is not an eigenvalue of $M$,  by some easy computations, we have that
\begin{eqnarray*}
 \det(S\gamma_\alpha(T)-\lambda I_{2n}) &=& \det(MM(\alpha)-\lambda I_{2n}) \nonumber \\  &=&  \det(M-\lambda I_{n}+\alpha M
M_1+\cdots+\alpha^k M M_k+\cdots) \nonumber \\  &=& \det(M-\lambda
I_{n})\det(I_n+\cdots+\alpha^k(M-\lambda I_{n})^{-1}M M_k+\cdots).
\label{c4.19}
\end{eqnarray*}
Let \begin{eqnarray}G_k=(M-\lambda I_{n})^{-1}M M_k,\end{eqnarray}
and
$$f(\alpha)=\det(I_n+\cdots+\alpha^kG_k+\cdots), $$ which is an analytic
function  on $\mathbb C$. Next, we will compute the Taylor expansion for $f(\alpha)$.
Let $G(\alpha)=\sum\limits_{k=1}^\infty\alpha^{k-1}G_k$, then
for $\alpha$ small enough, by Theorem \ref{thm2.4}, we have
\begin{eqnarray} f(\alpha) &=& \det(I_n+\alpha G(\alpha)) \nonumber \\  &=& \exp\Big(\sum_{m=1}^\infty\frac{(-1)^{m+1}}{m}\alpha^m Tr\big(G(\alpha)^m\big)\Big) \nonumber \\  &=&
\exp\Big(\sum_{m=1}^\infty\frac{(-1)^{m+1}}{m}\alpha^m
Tr\Big[\Big(\sum_{k=1}^\infty \alpha^{k-1}G_k\Big)^m\Big]\Big) \nonumber \\  &=&
\exp\Big(\sum_{m=1}^\infty\frac{(-1)^{m+1}}{m}\Big[\sum_{k_1,\cdots,k_m=1}^\infty\alpha^{k_1+\cdots+k_m}Tr(G_{k_1}\cdots
G_{k_m})\Big]\Big).
 \label{cc4.21}  \end{eqnarray}
Since $f(\alpha)$ vanishes nowhere near $0$, we can write  $f(\alpha)=e^{g(\alpha)}$, then by (\ref{cc4.21}), some direct computation  shows that
\begin{eqnarray}  g^{(m)}(0)/m != \sum_{k=1}^m \frac{(-1)^{k+1}}{k}\Big(\sum_{j_1+\cdots+j_k=m}Tr(G_{j_1}\cdots
G_{j_k})\Big). \label{cc4.22}  \end{eqnarray}
For $\alpha$
small enough, let $g(\alpha)$ be the function satisfying  \begin{eqnarray}
\det(S\gamma_\alpha(T)-\lambda
I_{n})=\det(M-\lambda I_{n})\cdot \exp(g(\alpha))
\label{cc4.29},
\end{eqnarray}
then the coefficients  $g^{(k)}(0)/k!$ could be determined  by
(\ref{cc4.22}). And we have the following theorem, which is the main result in this subsection.

\begin{thm}\label{thm1.1b}
Under the above assumption, let $g(\alpha)$ be the function in
(\ref{cc4.29}). Let $g(\alpha)=\sum\limits_{m=1}^\infty c_m
\alpha^m$ be its Taylor expansion. Then
\begin{eqnarray}
c_m=\sum_{k=1}^m \frac{(-1)^{k+1}}{k}\Big(\sum_{j_1+\cdots+j_k=m}Tr(G_{j_1}\cdots
G_{j_k})\Big). \label{t1}
\end{eqnarray}
where $G_k=(M-\lambda I_{n})^{-1}M M_k$, and $M=S\gamma_0(T)$,
\begin{eqnarray}
M_k=\int_0^T\hat{D}(t_1)\int_0^{t_1}\hat{D}(t_2)\cdots\int_0^{t_{k-1}}\hat{D}(t_k)dt_k\cdots dt_2dt_1,
k\in\mathbb N.
\end{eqnarray}
\end{thm}
 We only list the first $4$
terms
\begin{eqnarray*} g^{(1)}(0)=Tr(G_1) \label{cc4.23},
\end{eqnarray*}
\begin{eqnarray*} g^{(2)}(0)/2=Tr(G_2)-\frac{1}{2}Tr(G_1^2) \label{cc4.24},
\end{eqnarray*}
\begin{eqnarray*} g^{(3)}(0)/3!=Tr(G_3)-Tr(G_1G_2)+\frac{1}{3}Tr(G_1^3) \label{cc4.25},
\end{eqnarray*}
\begin{eqnarray*} g^{(4)}(0)/4!=Tr(G_4)-\frac{1}{2}Tr(G_2^2)-Tr(G_1G_3)+Tr(G_1^2G_2)-\frac{1}{4}Tr(G_4)
\label{cc4.26}.
\end{eqnarray*}
By the definition of $G_k$,
\begin{eqnarray*}Tr(G_1)=Tr(M_1M(M-\lambda I_{2n})^{-1})=Tr \Big(\int_0^T\hat{D}(s)ds \cdot M(M-\lambda I_{2n})^{-1} \Big)  \label{cc4.27},
\end{eqnarray*}
\begin{eqnarray*}Tr(G_2)=Tr(M_2M(M-\lambda I_{2n})^{-1})=Tr \Big(\int_0^T\hat{D}(s)\int_0^s\hat{D}(\sigma)d\sigma ds \cdot M(M-\lambda I_{2n})^{-1} \Big)
\label{cc4.28}.
\end{eqnarray*}
Generally, \begin{eqnarray*}
Tr(G_k^m)=Tr\Big(\Big[\int_0^T\hat{D}(t_1)\int_0^{t_1}\hat{D}(t_2)\cdots\int_0^{t_{k-1}}\hat{D}(t_k)dt_k\cdots dt_2 dt_1\cdot
M(M-\lambda I_{2n})^{-1}\Big]^m\Big)\label{cc4.28},  \end{eqnarray*} and
$Tr(G_{j_1}\cdots G_{j_k}) $  could be given similarly.

In the case of Hamiltonian,  there are some symmetry property of $M_k$, please refer \cite{HOW} for the detail.

\subsection{Trace formula for linearly parameterized first order ODE}
For $D_\alpha=D_0+\alpha D$, let $D_\alpha$ take place of $D$ in both sides of Hill-type formula (\ref{eq0.1a}), we get
\begin{eqnarray}
&&\det\left[\left(\d-D_\alpha+\nu I_n\right)\left(\d+\hat{P}_0\right)^{-1}\right]\nonumber\\&&=(-1)^n|C(S)| e^{-\frac{n\nu T}{2}} e^{-\frac{1}{2}\int_0^T Tr(D_\alpha)dt} \det(S\gamma_{D_\alpha}(T)-e^{\nu T}I_n). \label{thmf1}
\end{eqnarray}
As we have proved that, both sides of (\ref{thmf1}) are analytic functions on $\alpha$.
 Notice that the left hand side
\begin{eqnarray}
\left(\d-D_\alpha+\nu I_n\right)\left(\d+\hat{P}_0\right)^{-1}&=&\left(\d-D_0+\alpha D+\nu I_n\right)\left(\d- D_0+\nu I_n\right)^{-1}\nonumber\\
                                                              &&\ \ \cdot\left(\d-D_0+\nu I_n\right)\left(\d+\hat{P}_0\right)^{-1}.
\end{eqnarray}
Hence
\begin{eqnarray}
\det\left[\left(\d-D_\alpha+\nu I_n\right)\left(\d+\hat{P}_0\right)^{-1}\right]&=&\det\left[\left(\d- D_0-\alpha D+\nu I_n\right)\left(\d- D_0+\nu I_n\right)^{-1}\right]\nonumber\\
&&\ \ \cdot\det\left[\left(\d- D_0+\nu I_n\right)\left(\d+\hat{P}_0\right)^{-1}\right]\nonumber\\
&=&\det\left(id-\alpha D\left(\d- D_0+\nu I_n\right)^{-1}\right)\nonumber\\ &&\ \ \cdot\det\left[\left(\d- D_0+\nu I_n\right)\left(\d+\hat{P}_0\right)^{-1}\right]
\end{eqnarray}
Let $$F=D\left(\d- D_0+\nu I_n\right)^{-1},$$
then the left hand side of (\ref{thmf1})
 \begin{eqnarray}
 f(\alpha)&=&\det\left[\left(\d-D_\alpha+\nu I_n\right)\left(\d+\hat{P}_0\right)^{-1}\right]\nonumber\\
 &=&\det\left[\left(\d-D_\alpha+\nu I_n\right)\left(\d- D_0+\nu I_n\right)^{-1}\right]\det\left[\left(\d- D_0+\nu I_n\right)\left(\d+\hat{P}_0\right)^{-1}\right]\nonumber\\
 &=&\det(id-\alpha F)\cdot \det\left[\left(\d- D_0+\nu I_n\right)\left(\d+\hat{P}_0\right)^{-1}\right].
 \end{eqnarray}
 Notice that $F$ is a Hilbert-Schmidt operator with trace finite condition with respect to $\hat{P}_k$. By Theorem \ref{thm2.4}, for $\alpha$ small,
\begin{eqnarray}
f(\alpha)=\exp\left(\sum\limits_{m=1}\frac{-1}{m}\alpha^m Tr(F^m)\right)\cdot \det\left[\left(\d-D_0+\nu I_n\right)\left(\d+\hat{P}_0\right)^{-1}\right].\label{thmf1l}
\end{eqnarray}
On the other hand, by Theorem \ref{thm1.1b}, the right hand side of
(\ref{thmf1}) equals to
\begin{eqnarray}
&&(-1)^n|C(S)| e^{-\frac{n\nu T}{2}} e^{-\frac{1}{2}\int_0^T Tr(D_\alpha)dt} \det(S\gamma_{\alpha}(T)-e^{\nu T}I_n)\nonumber\\&&
=(-1)^n|C(S)| e^{-\frac{n\nu T}{2}} e^{-\frac{1}{2}\int_0^T Tr(D_0)dt}\det\left(M-\lambda I_n\right) e^{-\frac{\alpha}{2}\int_0^T Tr(D)dt}e^{g(\alpha)},\label{thmf1r}
\end{eqnarray}
where $g(\alpha)=\sum\limits_{m=1}^\infty c_m\alpha^m$ satisfies $$ \det(S\gamma_{\alpha}(T)-e^{\nu T}I_n)=\det(M-\lambda I_n)e^{g(\alpha)},$$
 and $c_m$ are  given in (\ref{t1}).
Comparing  (\ref{thmf1l}) with (\ref{thmf1r}), we have
\begin{eqnarray*}
g(\alpha)=\frac{\alpha}{2}\int_0^T Tr(D)dt+\sum\limits_{m=1}\frac{-1}{m}\alpha^m Tr(F^m),
\end{eqnarray*}
and hence
\begin{eqnarray}
Tr(F)=\frac{1}{2}\int_0^T Tr ( D(t))dt-c_1
\label{tra.3},
\end{eqnarray}
and
\begin{eqnarray}
Tr(F^m)=-mc_m, \,\  m\geq 2
\label{tra.4}.
\end{eqnarray}
Thus we get
\begin{thm}\label{thm4.1a}
Let $\nu\in \mathbb C$ such that $\d-{D}_0-\nu $ is invertible,  $F=D\left(\d-D_0+\nu I_n\right)^{-1}$,  then
\begin{eqnarray}
Tr(F)=\frac{1}{2}\int_0^T Tr ( D(t))dt-Tr(G_1)
\label{tra.5},
\end{eqnarray}
and for any positive integer $m\geq2$,
\begin{eqnarray}\label{0.0.0}
Tr(F^m)=m\sum_{k=1}^m
\frac{(-1)^{k}}{k}\Big[\sum\limits_{j_1+\cdots+j_k=m}Tr(G_{j_1}\cdots
G_{j_k})\Big].
\end{eqnarray}
where $G_k=(M-\lambda I_{2n})^{-1}M M_k$, and $M=S\gamma_0(T)$,
\begin{eqnarray}
M_k=\int_0^T\hat{D}(t_1)\int_0^{t_1}\hat{D}(t_2)\cdots\int_0^{t_{k-1}}\hat{D}(t_k)dt_k\cdots dt_2dt_1,
k\in\mathbb N.
\end{eqnarray}

\end{thm}

For large $m$, the right hand side of (\ref{0.0.0}) is a little
complicated. However, for $m=1,2$, we can write it down more
precisely.
\begin{cor}\label{cor4.1a}Under the assumption as in Theorem \ref{thm4.1a},
\begin{eqnarray*}\label{eq3.49a}
Tr(F)= \frac{1}{2}\int_0^T Tr ( D(t))dt-Tr\Big(\int_0^T\gamma_{0}^{-1}(t)D(t)
\gamma_{0}(t)dt\cdot M(M-e^{\nu T} I_{n})^{-1}\Big).
\end{eqnarray*}
and
\begin{eqnarray*}
Tr(F^2)&=&-2Tr \Big(\int_0^T\gamma_{0}^{-1}(t)D(t) \gamma_{0}(t)\int_0^t\gamma_{0}^{-1}(s)D(s) \gamma_{0}(s)ds dt \cdot M(M-\lambda I_{n})^{-1} \Big)\nonumber\\
                                &&\ \ \ +Tr\Big(\Big[\int_0^T\gamma_{0}^{-1}(t)D(t) \gamma_{0}(t)dt\cdot M(M-e^{\nu T} I_{n})^{-1}\Big]^2\Big).
\end{eqnarray*}
\end{cor}

Especially \bea Tr\left(D\left(\d+\nu I_n\right)^{-1}\right)&=& Tr\Big(\int_0^TD(t)dt\cdot(\frac{1}{2}- S(S-e^{\nu T} I_{n})^{-1})\Big) \nonumber \\ &=& -\frac{1}{2}Tr\Big(\int_0^TD(t)dt\cdot( S+e^{\nu T})(S-e^{\nu T} I_{n})^{-1})\Big). \label{4.e1} \eea
Taking derivative  on both sides of (\ref{4.e1}), we get
 \bea Tr\left(D\left(\d+\nu I_n\right)^{-2}\right)=-Te^{\nu T} Tr\Big(\int_0^TD(t)dt\cdot S(S-e^{\nu T} I_{n})^{-2})\Big).\label{4.e2}  \eea

\section{Examples}

 Krein considered second order system in \cite{Kr1}
 \bea y'' +\lambda R(t)y=0,\,\
y(0)+y(T)=y'(0)+y'(T)=0,\label{k1} \eea where $R(t)$ is continuous
path of real symmetric matrices on $\mathbb{R}^n$.
   Set
$R_{ave}=\frac{1}{T}\int_0^TR(t)dt$ and
 \bea  X(t)=\int_0^t(R(s)-R_{ave})ds+C, \label{k2} \eea where $C$
is a constant matrix which is chosen such that $X_{ave}=0$.
Let $\lambda_j$ be the eigenvalues of (\ref{k1}),
Krein    get  \bea
\sum\frac{1}{\lambda_j}=\frac{T}{4}\int_0^TTr(R(t))dt, \label{k3}
\eea
and \bea
\sum\frac{1}{\lambda_j^2}=\frac{T}{2}\int_0^TTr(X^2(t))dt+\frac{T^2}{48}Tr[(\int_0^TR(t)dt)^2].
\label{k4} \eea

In this section, we will give a generalization of Krein's trace formula from our viewpoint. To simplify the notation, let
$\mathcal{A}(\nu)=-(\d+\nu)^2$ and   denote by $$ R_{ave}=\frac{1}{T}\int_0^TR(t)dt, $$
 which is a constant matrix.  From (\ref{4.e2}),  we have
\bea   Tr (R\mathcal{A}(\nu)^{-1})=-\omega T^2\cdot Tr(R_{ave}\cdot
S(S-\omega)^{-2}), \label{aa.1}\eea
 By taking derivative with respect
to $\nu$  on both sides of (\ref{aa.1}), we get
\bea   Tr (R\mathcal{A}(\nu)^{-2})=\frac{\omega T^4}{6}Tr(R_{ave} S(S^2+4\omega S+\omega^2)(S-\omega)^{-4}). \label{aa.2}\eea
Especially, if  $\int_0^T Rdt=0$, then \bea   Tr
(R\mathcal{A}(\nu)^{-2})=0. \label{aa.2.1}\eea Please note that
(\ref{aa.1}) is obtained by taking derivative on   both sides of
(\ref{4.e1}).
 Obviously, in this case,  we could calculate $ Tr\(D\(\d+\nu I_n\)^{-k}\)$ for any
 $k\in\mathbb{N}$ by taking derivative with respect to $\nu$ for both sides of (\ref{4.e1}).

The purpose of the remaining part of this subsection is to compute $Tr\[\(R\mathcal{A}(\nu)^{-1}\)^2\]$. In fact, we will use our trace formula (\ref{0.0.0}) to express $Tr\[\(R\mathcal{A}(\nu)^{-1}\)^2\]$ as the form of multiple integral.
Inspired by Krein \cite{Kr1}, let  $$ X(t)=\int_0^t(R(s)-R_{ave})ds+C, $$ where $C$ is a constant matrix.  Obviously, $$X(0)=X(T)=C.$$ At first, we will calculate $Tr\[\((R-R_{ave})\mathcal{A}(\nu)^{-1}\)^2\]$.
\begin{prop}\label{prop4.1} For any constant matrix $C$ which satisfies $CS=SC$, $X(t)$ is defined as above,  then
\bea
&&Tr\[\((R-R_{ave})\mathcal{A}(\nu)^{-1}\)^2\]\nonumber\\&&=-2\omega Tr\left[T\int_0^T X^2 dt\cdot S(S-\omega )^{-2}\right]+2Tr\left[\(\int_0^TX(t)dt \cdot S(S-\omega)^{-1}\)^2\right]\nonumber \\&&\ \ \ \ -4Tr\[\int_0^TX(t)\int_0^tX(s)dsdt\cdot S(S-\omega)^{-1}\].    \label{aa.3}
\eea
\end{prop}

\begin{proof}
Recall that $\mathcal{A}(\nu)=-(\d+\nu)^2$. Please note that for any bounded operators $A,B$ such that $AB$ is trace class operator,  we have that $Tr(AB)=Tr(BA)$. Hence
\bea Tr\[\((R-R_{ave})\mathcal{A}(\nu)^{-1}\)^2\]&=&Tr\left[(R-R_{ave})\(\d+\nu\)^{-2}(R-R_{ave})\(\d+\nu\)^{-2}\right]\nonumber
\\&=&Tr\left[\(\(\d+\nu\)^{-1}(R-R_{ave})\(\d+\nu\)^{-1}\)^2\right]\label{aa.1a}.
\eea
Noting that $ X(t)=\int_0^t(R(s)-R_{ave})ds+C $ for some constant matrix $C$, we have
\bea \(\d+\nu\)X-X\(\d+\nu\)=\dot{X}=R-R_{ave}, \eea
thus
 \bea X\(\d+\nu\)^{-1}-\(\d+\nu\)^{-1}X=\(\d+\nu\)^{-1}(R-R_{ave})\(\d+\nu\)^{-1}.\label{aa.1b} \eea
By (\ref{aa.1a}) and (\ref{aa.1b}),
\bea Tr\[\((R-R_{ave})\mathcal{A}(\nu)^{-1}\)^2\]&=& Tr\[\(X\(\d+\nu\)^{-1}-\(\d+\nu\)^{-1}X\)^2\] \nonumber \\ &=& 2Tr\[\(X\(\d+\nu\)^{-1}\)^2\]-2Tr\left[X^2\(\d+\nu\)^{-2} \right]
\nonumber \\ &=& 2Tr\left[\(X\(\d+\nu\)^{-1}\)^2\right]+2Tr\left[X^2\mathcal{A}(\nu)^{-1}\right].\label{aa.1c}
\eea
From $(\ref{aa.1})$
\bea  Tr\[X^2\mathcal{A}(\nu)^{-1}\]=-\omega TTr\left[\int_0^T X^2dt\cdot S(S-\omega)^{-2}\right]. \label{aa.2a}\eea
To continue, we should calculate $Tr\left[\(X\(\d+\nu\)^{-1}\)^2\right]$ by using Corollary \ref{cor4.1a}. In this case, $D_0$ in Corollary \ref{cor4.1a} is $0$, thus $\gamma_0(t)=I_n$, $M=S$. It follows that
\bea
&&Tr\left[\(X\(\d+\nu\)^{-1}\)^2\right]\nonumber\\&&=Tr\left[\(\int_0^TX(t)dt \cdot S(S-\omega)^{-1}\)^2\right]-2Tr\left[\int_0^TX(t)\int_0^tX(s)dsdt\cdot S(S-\omega)^{-1}\right]. \label{aa.2b} \eea
By substituting (\ref{aa.2b}) and (\ref{aa.2a}) into (\ref{aa.1c}), we have the desired result.
 \end{proof}

Please note that if $S=\pm I_n$, then we have
\bea Tr\left[\int_0^TX(t)\int_0^tX(s)ds dt\cdot S(S-\omega)^{-1}\right]=\frac{1}{2}Tr\left[\(\int_0^TX(t)dt\)^2\cdot S(S-\omega)^{-1}\right] \label{aa.2c} \eea
Moreover, in this case,  the constant matrix  $C$  in Proposition \ref{prop4.1} could be chosen arbitrary. Particularly, $C$ could be chosen such that $\int_0^T X(t)dt=0$.
\begin{cor} In the case $S=\pm I_n$, if we choose $C$ such that $\int_0^TX(t)dt=0$, then we have
\bea    Tr\left[\((R-R_{ave})\mathcal{A}(\nu)^{-1}\)^2\right]=-\frac{2\omega}{(1\mp \omega)^2} Tr\left[T^2\int_0^T X^2 dt\right]. \label{aa.2d}\eea
\end{cor}

\begin{prop}\label{pro6.7} Suppose $R_{ave}S=SR_{ave}$, then
 \bea
 &&Tr((R\mathcal{A}(\nu)^{-1})^2)\nonumber\\&&=\frac{\omega T^4}{6}Tr(R^2_{ave} S(S^2+4\omega S+\omega^2)(S-\omega)^{-4})-2\omega T\cdot Tr\left(\int_0^T X^2 dt\cdot S(S-\omega )^{-2}\right)\nonumber \\ && \ \ \  +2Tr\left(\left(\int_0^TX(t)dt \cdot S(S-\omega)^{-1}\right)^2\right) -4Tr\left(\int_0^TX(t)\int_0^tX(s)ds\cdot S(S-\omega)^{-1}\right)
 \eea
\end{prop}

\begin{proof}
\bea
 \(R\mathcal{A}(\nu)^{-1}\)^2 &=& ((R_{ave}+(R-R_{ave}))\mathcal{A}(\nu)^{-1})^2 \nonumber \\
&=&\((R-R_{ave})\mathcal{A}(\nu)^{-1}\)^2+\(R_{ave}\mathcal{A}(\nu)^{-1}\)^2\nonumber \\
&&+(R-R_{ave})\mathcal{A}(\nu)^{-1}R_{ave}\mathcal{A}(\nu)^{-1}+R_{ave}\mathcal{A}(\nu)^{-1}(R-R_{ave})\mathcal{A}(\nu)^{-1}.  \label{aa.6}
 \eea
Please note that $R_{ave}S=SR_{ave}$ implies that $\mathcal{A}(\nu)^{-1}$ commutes with $R_{ave}$. We have
\bea Tr (R\mathcal{A}(\nu)^{-1})^2) &=& Tr(((R-R_{ave})\mathcal{A}(\nu)^{-1})^2)+Tr(R_{ave}\mathcal{A}(\nu)^{-1})^2)\nonumber \\
&&+2Tr((R-R_{ave})R_{ave}\mathcal{A}(\nu)^{-2}).  \label{aa.6} \eea
Direct computation shows that
\bea Tr(R_{ave}\mathcal{A}(\nu)^{-1})^2) &=& Tr(R^2_{ave}\mathcal{A}(\nu)^{-2})\nonumber \\
&=& \frac{\omega T^4}{6}Tr(R^2_{ave} S(S^2+4\omega S+\omega^2)(S-\omega)^{-4}) . \label{eq6.2} \eea
Since $\int_0^T (R-R_{ave})dt=0$, by (\ref{aa.2.1}) we have
\bea
 Tr\[(R-R_{ave})R_{ave}\mathcal{A}(\nu)^{-2}\]=0.  \label{eq6.3}
\eea
Combining (\ref{aa.6}) with (\ref{aa.3}), (\ref{eq6.2}) and (\ref{eq6.3}), the desired result is proved.
 \end{proof}

\begin{cor}\label{cor6.8} In the case $S=\pm I_n$, if we choose $C$ such that $\int_0^TX(t)dt=0$, then we have
\bea   &&  Tr((R\mathcal{A}(\nu)^{-1})^2)\nonumber\\&&=\frac{\pm(1\pm 4\omega+\omega^2)\omega T^2}{6(1\mp \omega)^{4}}Tr\left[\left(\int_0^T R(t)dt\right)^2\right]\mp\frac{2\omega T}{(1\mp \omega)^2}\cdot \left(\int_0^T Tr( X^2) dt\right). \label{aa.2e}\eea
\end{cor}

\begin{rem}
More specially, for the case that Krein considered, that is, let $S=-I_n$ and $\nu=0$, in this case, $\omega=1$. By (\ref{aa.1}), we have Krein's trace formula (\ref{k3}). Moreover, by (\ref{aa.2e}), we have
\begin{eqnarray*}
\sum\frac{1}{\lambda_j^2}=\frac{T}{2}\int_0^TTr(X^2(t))dt+\frac{T^2}{48}Tr\[\(\int_0^TR(t)dt\)^2\],
\end{eqnarray*}
which is Krein's trace formula (\ref{k4}).  \end{rem}

\medskip

\noindent {\bf Acknowledgements.} We would like to thank Yiming Long for his valuable suggestions and encouragements.


\begin{thebibliography}{99}

\bibitem{De} R. Denk, On Hilbert-Schmidt operators and determinants corresponding
to periodic ODE systems, Differential and integral operators
(Regensburg, 1995), 57-71, Oper. Theory Adv. Appl., Vol.102,
Birkh\"auser, Basel, 1998.

\bibitem{Hi} G. W. Hill, On the part of the motion of the lunar perigee which is a function of the mean motions of the sun and moon, Cambridge, Wilson, 1877; reprinted with some additions at Acta Math. Vol. 8(1886), 1-36.


\bibitem{HO} X. Hu, Y. Ou, An Estimation for the Hyperbolic Region of Elliptic Lagrangian Solutions in the Planar Three-body Problem.  Regular and Chaotic Dynamics,  Vol. 18, No. 6,  (2013),  pp. 732-741.


\bibitem{HOW} X. Hu, Y. Ou and P. Wang, Trace formula for linear Hamiltonian systems with its applications to elliptic Lagrangian solutions,  Arch. Ration. Mech. Anal. 216 (2015), no. 1, 313-357.


\bibitem{HW1} X. Hu and P. Wang, Conditional Fredholm determinant for the S -periodic orbits in Hamiltonian systems. J. Funct. Anal. 261 (2011), no. 11, 3247-3278.

\bibitem{Kr1} M.G. Krein, On tests for the stable boundedness of solutions of periodic canonical systems, PrikL Mat. Mekh., 19, Issue 6,
641-680 (1955).

\bibitem{Kr2} M.G. Krein, Foundation of the theory of $\lambda$-zones
  of stability of a canonical systems of linear differential
  equations with periodic coefficients, In Memoriam: A.A.Andronov,
  Izdat.Akad.Nauk SSSR, Moscow,1955,pp.413-498

\bibitem{Lon4} Y. Long, Index Theory for Symplectic Paths with
Applications, Progress in Math., Vol.207, Birkh\"auser. Basel. 2002.

\bibitem{Po} A. Poincar\'{e}, Sur les d\'{e}terminants d'ordre infini, Bull. Soc. math. France, 14 (1886), 77-90.


\bibitem{Si} B. Simon, Trace ideals and their applications. Second edition. Mathematical Surveys and Monographs, 120. American Mathematical Society, Providence, RI, 2005.

\end{thebibliography}
\end{document}